\documentclass[reqno, 12pt]{amsart}
\pdfoutput=1
\makeatletter
\let\origsection=\section \def\section{\@ifstar{\origsection*}{\mysection}} 
\def\mysection{\@startsection{section}{1}\z@{.7\linespacing\@plus\linespacing}{.5\linespacing}{\normalfont\scshape\centering\S}}
\makeatother        
\usepackage{amsmath,amssymb,amsthm}    
\usepackage{mathabx}\usepackage{mathrsfs}
\usepackage{bbm, accents} 

\usepackage{xcolor}  	
\usepackage[backref]{hyperref}
\hypersetup{
	colorlinks,
    linkcolor={red!60!black},
    citecolor={green!60!black},
    urlcolor={blue!60!black},
}

\usepackage{bookmark}

\usepackage[abbrev,msc-links,backrefs]{amsrefs} 
\usepackage{doi}

\renewcommand{\PrintDOI}[1]{\doi{#1}}

\usepackage[T1]{fontenc}
\usepackage{lmodern}

\usepackage[english]{babel}
\usepackage{microtype}

\numberwithin{equation}{section}

\theoremstyle{plain}
\newtheorem{thm}{Theorem}[section]
\newtheorem{lemma}[thm]{Lemma}

\theoremstyle{definition}
\newtheorem{dfn}[thm]{Definition}
\newtheorem{quest}[thm]{Question}
\newtheorem{rem}[thm]{Remark}

\usepackage{geometry}
\let\eps=\varepsilon
\let\theta=\vartheta
\let\rho=\varrho
\let\phi=\varphi

\let\lra=\longrightarrow

\def\st{\,|\,}
\def\bl{\bigl(}
\def\br{\bigr)}

\usepackage{enumitem}

\begin{document}
\title[Counting odd cycles]{Counting odd cycles in locally dense graphs}
\author{Christian Reiher}
\address{Fachbereich Mathematik, Universit\"at Hamburg,
   D-20146 Hamburg, Germany}
\email{Christian.Reiher@uni-hamburg.de}


\begin{abstract}
We prove that for any given $\eps>0$ and $d\in[0,1]$, every sufficiently large 
\hbox{$(\eps, d)$-dense} graph $G$ contains for each odd integer $r$ at least 
$(d^r-\eps)|V(G)|^r$ cycles of length~$r$. Here, $G$ being $(\eps, d)$-dense means that 
every set $X$ containing at least~$\eps\,|V(G)|$ vertices spans at least 
$\tfrac d2\, |X|^2$ edges, and what we really count is the number of 
homomorphisms from an $r$-cycle into $G$. 

The result adresses a question of {\sc Y. Kohayakawa}, {\sc B. Nagle}, {\sc V. R\"{o}dl}, 
and {\sc M.~Schacht}. 
\end{abstract}

\maketitle

\section{Introduction}\label{sec:intro}

A graph $G$ is said to be $d$-quasirandom for some real number $d\in[0, 1]$ 
if each subset~$X$ of the set $V$ of its vertices spans $\tfrac d2\,|X|^2+o(|V|^2)$ edges. 
A part of the reason as to why this concept is a useful one is that 
it is known that into any such graph there are $\bl d^{|E(H)|}+o_H(1)\br|V(G)|^{|V(H)|}$ 
homomorphisms  from any fixed graph $H$. 

Recently {\sc Y. Kohayakawa}, {\sc B. Nagle}, {\sc V. R\"{o}dl}, and 
{\sc M. Schacht}~\cite{KNRS} asked whether a certain variant of this implication, 
which would in some situations be stronger, is also true. 
Namely, if we demand from $G$ only that any $X$ as above spans at least 
$\tfrac d2\,|X|^2+o(|V|^2)$ edges, so that locally $G$ 
may have far more edges than a quasirandom graph would have, does it then still follow 
that one has at least $(d^{|E(H)|}+o_H(1))|V(G)|^{|V(H)|}$ homomorphisms from $H$ into $G$?

To get these ideas more precise one may make the following

\begin{dfn}  
Let $\eps\in(0, 1)$ and $d\in [0, 1]$ be given. A graph $G$ on $n$ vertices is said 
to be {\it $(\eps, d)$-dense} if each $X\subseteq V(G)$ with $|X|\ge\eps n$ spans 
at least $\tfrac d2\,|X|^2$ edges.  
\end{dfn}

Then what was asked in~\cite{KNRS} is this:

\begin{quest}\label{Frage}
For which graphs $H$ is it true that for each $\delta>0$ there exists an $\eps>0$ 
such that there are at least $(d^{|E(H)|}-\delta)|V(G)|^{|V(H)|}$ homomorphisms 
of $H$ into any sufficiently large graph $G$ that happens to be $(\eps, d)$-dense 
for some real $d$?
\end{quest}

It was observed in~\cite{KNRS} that the answer is affirmative if $H$ is a clique, 
a complete multipartite graph, the line graph of a boolean cube, or a bipartite graph 
that satisfies {\sc Sidorenko's} conjecture \cite{Sid}, which was formulated 
independently by {\sc Erd\H{o}s} and {\sc Simonovits} (see e.g.~\cite{Sim}). 
The last of these classes is known to contain all even cycles, and the authors 
of~\cite{KNRS} wondered explicitly about the case of odd cycles. 
The main result of this article addresses this problem. In order to be able to 
state it more briefly, we introduce the following

\begin{dfn}
Given a graph $G$ and an integer $r\ge 2$, we write $C_r(G)$ for the number of 
sequences $(x_1, x_2, \ldots, x_r)\in V(G)^r$ having the property that 
$x_1x_2, x_2x_3, \ldots, x_rx_1$ are edges of~$G$. 
\end{dfn}

\begin{thm}\label{Hauptsatz}
If a graph $G$ is $(\eps, d)$-dense and possesses at least $\frac 2{\eps-\eps^2}$ 
vertices, then we have $C_r(G)\ge (d^r-\eps)|V(G)|^r$
for each odd number $r\ge 3$.
\end{thm}

The proof is prepared in the next section by some lemma that has an analytic flavour, 
whilst the main proof is deferred to Section~\ref{sec:proof}. It uses an inequality 
due to {\sc G.~R. Blakley} and {\sc P.~A. Roy}~\cite{BR} that basically tells us that 
paths obey {\sc Sidorenko's} conjecture.

\begin{lemma}\label{Sido}
For any positive integer $k$ and any graph $G$ satisfying $|E(G)|\ge \tfrac d2\,|V(G)|^2$ 
there are at least $d^k|V(G)|^{k+1}$ homomorphisms from a path of length $k$ into $G$.
\end{lemma} 

\section{More on denseness} \label{sec:ana}

In this section we prove that the property of being dense does more or less imply a 
weighted version of itself. A precise statement along those lines reads as follows:

\begin{lemma}\label{f}
Suppose that $G$ is an $(\eps, d)$-dense graph on $n$ vertices. 
Then for every function $f\colon V(G)\lra[0,1]$ satisfying 
$\sum_{x\in V(G)}f(x)\ge \eps n$ we have
\[
	\sum_{xy\in E(G)}f(x)f(y)\ge \tfrac d2\Bigl(\sum_{x\in V(G)}f(x)\Bigr)^2-n\,.
\]
\end{lemma}

\begin{proof}
1. The space of all functions $f\colon V(G)\lra [0,1]$ satisfying 
$\sum_{x\in V(G)}f(x)\ge \eps n$ is compact and consequently it contains 
a member $f_0$ for which the continuous expression
\[
	\sum_{xy\in E(G)}f(x)f(y)- \tfrac d2 \Bigl(\sum_{x\in V(G)}f(x)\Bigr)^2
\]
attains its least possible real value $\Omega$, and for which subject to this the set 
$X$ of all $x\in V(G)$ with $f(x)\in \{0,1\} $ has its maximal possible size. 
Evidently it suffices to prove $\Omega\ge -n$.

\medskip

2. As a first step in this direction we will verify $|V(G)-X|\le 1$. 
Assume contrariwise that there are two distinct vertices $x$ and $y$ belonging to $V(G)-X$. 
Let $\eta$ be any real number whose absolute value is so small that 
$f_0(x)\pm\eta, f_0(y)\pm\eta\in[0,1]$. 
For any two real numbers $a$ and $b$ we write $H(a,b)$ for the value that
\[
	\sum_{xy\in E(G)}f(x)f(y)- \tfrac d2 \Bigl(\sum_{x\in V(G)}f(x)\Bigr)^2
\]
attains for the function $f\colon V(G)\lra [0,1]$ given by
\[
	f(z)=\begin{cases}
			a & \text{ if } z=x\\
			b &\text{ if } z=y\\
			f_0(z) & \text{ if } z\ne x, y. 
  			\end{cases}
\]
So for example $H\bigl(f_0(x), f_0(y)\bigr)=\Omega$. 
If $xy$ were an edge of $G$, then there would have to exist real numbers 
$A$, $B$, $C$, and $T$ not depending on $\eta$ such that
\begin{align*}
	H\bl f_0(x)+\eta, f_0(y)-\eta\br 
	&=\bl f_0(x)+\eta\br \bl f_0(y)-\eta\br
	+A\bl f_0(x)+\eta\br + B\bl f_0(y)-\eta\br +C \\
	& =H\bl f_0(x), f_0(y)\br +\bl f_0(y)-f_0(x)+A-B\br \eta-\eta^2 \\
	& =\Omega+T\eta-\eta^2\,.
\end{align*}
In each of the three cases $T>0$, $T=0$, and $T<0$ it is easy to choose $\eta$ 
in such a way that $H\bl f_0(x)+\eta, f_0(y)-\eta\br <\Omega$ holds. 
As this contradicts the supposed minimality of $f_0$, we have thereby shown 
that $xy$ cannot be an edge of $G$. 
This, however, means that there exist three real numbers $A$, $B$, and $C$ not depending 
on $\eta$ such that
\begin{align*}
	H\bl f_0(x)+\eta, f_0(y)-\eta\br
	&=A\bl f_0(x)+\eta\br + B \bl f_0(y)-\eta\br +C \\
	& =H\bl f_0(x), f_0(y)\br +(A-B)\eta\,.
\end{align*}
Now the same contradiction arises as before unless $A=B$, in which case we have 
\[
	H\bl f_0(x)+\eta, f_0(y)-\eta\br = H\bl f_0(x), f_0(y)\br
\]
for any $\eta$. Thus taking 
\[
	\eta=\min\bl 1-f_0(x), f_0(y)\br
\]
we get a contradiction to the extremal choice of $X$. 
We have thereby learned that there is indeed some vertex $z\in V(G)$ with 
$V(G)-\{z\}\subseteq X$.

\medskip

3. Let us now define $\delta=f_0(z)$ and $A=\{x\in X\st f_0(x)=1\}\cup\{z\}$.
Evidently we have 
\[
	|A|\ge \sum_{x\in V(G)}f_0(x)\ge\eps n
\]
and, since $G$ is $(\eps, d)$-dense by hypothesis, 
this implies $e(A) \ge \tfrac 12\cdot d\,|A|^2$, 
where $e(A)$ refers to the number of edges spanned by $A$. 
Notice that 
\[
	\sum_{x\in V(G)}f_0(x)=|A|-(1-\delta)
\]
and
\[
	\sum_{xy\in E(G)}f_0(x)f_0(y)=e\bl A-\{z\}\br +\delta N=e(A)-(1-\delta)N\,,
\]
where $N$ denotes the number of edges from $z$ to $A-\{z\}$. 
Thereby we obtain indeed
\begin{align*}
	\Omega &= e(A) -(1-\delta)N-\tfrac d2 \bl |A|-(1-\delta)\br^2 \\
	&= \bl e(A)- \tfrac d2 |A|^2\br -(1-\delta)(N-d |A|)-\tfrac d2(1-\delta)^2 \\
	&\ge -N-1\ge -n\,,
\end{align*}
which concludes the proof. 
\end{proof}

\section{The proof of Theorem~\ref{Hauptsatz}} \label{sec:proof}

In this section we shall finally prove Theorem~\ref{Hauptsatz}. 
Write $r=2m+1$ with some positive integer $m$. Given two vertices $x$ and $y$ of 
$G=(V, E)$ we denote the number of sequences $(a_0, \ldots, a_m)\in V^{m+1}$ 
satisfying $a_0=x$, $a_0a_1, \ldots, a_{m-1}a_m\in E$, and $a_m=y$ by $q(x,y)$. 
Clearly we have $q(x,y)\le n^{m-1}$, where $n=|V|$, and
\[
C_r(G)=\sum_{(x, y, z)\in V^3; yz\in E}q(x, y)q(x, z)\,.
\]
Writing $Z$ for the set of all $x\in V$, for which $\sum_{y\in V} q(x, y)\ge\eps n^{m}$ 
holds, we obtain, by Lemma~\ref{f},
\begin{align*}
	C_r(G)&\ge 2n^{2m-2}\sum_{x\in Z}
			\sum_{yz\in E}\frac{q(x, y)}{n^{m-1}}\cdot\frac{q(x, z)}{n^{m-1}} \\
	&\ge \sum_{x\in Z} d\left(\sum_{y\in V} q(x, y)\right)^2-2n^{2m-1} \\
	&\ge d\sum_{x\in V}\left(\sum_{y\in V} q(x, y)\right)^2-d\eps^2n^{2m+1}-2n^{2m}\,. 
\end{align*}
Due to $n\ge\frac 2{\eps-\eps^2}$ we have $d\eps^2n+2\le\eps n$ and it follows that
\[
	C_r(G)\ge \frac dn\left(\sum_{x, y\in V}q(x, y)\right)^2-\eps n^{2m+1}\,. 
\]
In the light of Lemma~\ref{Sido} this entails
\begin{align*}
	C_r(G)&\ge \frac dn\left(\left(\frac{2|E|}{n^2}\right)^{m}
								\cdot n^{m+1}\right)^2-\eps n^{2m+1}\\
	&\ge d^{2m+1}n^{2m+1}-\eps n^{2m+1}=(d^r-\eps)n^r,
\end{align*}
thereby completing the proof of Theorem~\ref{Hauptsatz}. \hfill $\Box$

\begin{rem}
It should be clear that this method of proof allows to decide 
Question~\ref{Frage} in a few more cases.
\end{rem}

\subsection*{Acknowledgement} 
I would like to thank {\sc Mathias Schacht} for bringing 
Question~\ref{Frage} to my attention.

\end{document}